\documentclass{amsart}

\title     [Generalized $q$-boson algebras]{Generalized $q$-boson algebras and their integrable modules}
\thanks    {This work was partially supported by Grant-in-Aid for Scientific Research (C) 20540036, Japan Society for the Promotion of Science.}
\author    {Akira Masuoka}
\address   {Institute of Mathematics, University of Tsukuba, Ibaraki 305-8571, Japan}
\email     {akira@math.tsukuba.ac.jp}
\date      {}

\usepackage{amsmath}
\usepackage{amssymb}
\usepackage{amsthm}

\numberwithin{equation}{section}

\newtheorem{theorem}              {Theorem}[section]
\newtheorem{lemma}       [theorem]{Lemma}
\newtheorem{proposition} [theorem]{Proposition}
\newtheorem{corollary}   [theorem]{Corollary}
\newtheorem{definition}  [theorem]{Definition}
\newtheorem{remark}      [theorem]{Remark}

\newcommand{\cmdrhu}       {\rightharpoonup}

\newcommand{\cmdA}         {\mathcal{A}}
\newcommand{\cmdB}         {\mathcal{B}}

\newcommand{\cmdH}         {\mathcal{H}}
\newcommand{\cmdI}         {\mathcal{I}}

\newcommand{\cmdO}         {\mathcal{O}}
\newcommand{\cmdS}         {\mathcal{S}}
\newcommand{\cmdYD}        {\mathcal{YD}}

\newcommand{\cmdbbA}       {\mathbb{A}}
\newcommand{\cmdbbD}       {\mathbb{D}}
\newcommand{\cmdbbN}       {\mathbb{N}}
\newcommand{\cmdbbQ}       {\mathbb{Q}}
\newcommand{\cmdbbZ}       {\mathbb{Z}}
\newcommand{\cmdbbm}       {\mathbf{m}}

\newcommand{\cmdg}         {\mathfrak{g}}

\newcommand{\cmdop}        {\mathrm{op}}
\newcommand{\cmdcop}       {\mathrm{cop}}

\newcommand{\cmdVec}       {\mathcal{V}\!\mbox{\it ec}}

\newcommand{\cmddeg}       {\mathop{\mathrm{deg}}}
\newcommand{\cmddiag}      {\mathop{\mathrm{diag}}}

\newcommand{\cmdid}        {\mathop{\mathrm{id}}}

\newcommand{\cmdtw}        {\mathop{\mathrm{tw}}}
\newcommand{\cmdIm}        {\mathop{\mathrm{Im}}}
\newcommand{\cmdEnd}       {\mathop{\mathrm{End}}}
\newcommand{\cmdHom}       {\mathop{\mathrm{Hom}}}
\newcommand{\cmdKer}       {\mathop{\mathrm{Ker}}}
\newcommand{\cmdMax}       {\mathop{\mathrm{Max}}}
\newcommand{\cmddotrtimes} {\mathop{\raisebox{0.2ex}{\makebox[0.86em][l]{${\scriptstyle>\mathrel{\mkern-4mu}\lessdot}$}}\raisebox{0.12ex}{$ \shortmid$}}}
\newcommand{\cmdarrow}     {\mathop{\longrightarrow}}

\begin{document}

\begin{abstract}
  We define the generalized $q$-boson algebra $ \cmdB$ associated to a pair of Nichols algebras and a skew pairing.
  We study integrable $ \cmdB$-modules,
  generalizing results by M. Kashiwara and T. Nakashima on integrable modules over a $q$-boson (Kashiwara) algebra.
\end{abstract}

\maketitle

\setcounter{section}{-1}

\section{Introduction}

As was realized by Andruskiewitsch and Schneider (see \cite{AS}),
the notion of Nichols algebras gives a sophisticated, `coordinate-free' viewpoint to study quantized enveloping algebras and associated objects.
This paper aims to clarify from that viewpoint what happens to $q$-boson (Kashiwara) algebras and their modules.

Let $ U_q $ denote the quantized enveloping algebra associated to a symmetrizable generalized Cartan matrix $ \cmdbbA $, and assume that $q$ is transcendental over $ \cmdbbQ $.
To study crystal bases of $ U_q $,
Kashiwara \cite{K} introduced the $q$-{\it boson} ({\it Kashiwara}) {\it algebra} $ B_q $, and proved that the minus part $ U_q^- $ is naturally an integrable left $ B_q $-module (the {\it Verma module}), and is in fact simple as a $ B_q $-module.
He also announced without proof that every integrable left $ B_q $-modules is isomorphic to the direct sum of some copies of $ U_q^- $; see \cite[Remark 3.4.10]{K}.
A proof of this fact was given later by Nakashima \cite{N},
who introduced the {\it extremal projector} for the purpose.
Tan \cite[Proposition 3.2]{Tan} formulated the same result in the generalized situation when $ \cmdbbA $ is a symmetrizable Borcherds-Cartan matrix, and gave a proof.
But, his proof and even formulation are incomplete, I think; see Remark \ref{4.6}.

As was shown in \cite{M1}, $ U_q $ can be constructed as a cocycle deformation of a simpler, graded Hopf algebra.
In \cite{M2}, this construction was generalized in the context of ({\it pre}-){\it Nichols algebras},
and as its special case, the construction of $ U_q $ by generalized quantum doubles,
as was given by Joseph \cite{J}, was explained; see also \cite{RS}.
Recall that this last construction works, when we are given data $ R, S, \tau $.
Here, $ R = \bigoplus_{n = 0}^{\infty} R ( n ) $, $ S = \bigoplus_{n = 0}^{\infty} S ( n ) $ are Nichols algebras (see \cite{AS}),
a sort of braided graded Hopf algebras, in the braided tensor categories $ {}_J^J \cmdYD $, $ {}_K^K \cmdYD $, respectively,
of Yetter-Drinfeld modules, where $J$, $K$ are (ordinary) Hopf algebras with bijective antipode.
As their bosonizations, we have the graded Hopf algebras $ R \cmddotrtimes J $, $ S \cmddotrtimes K $.
The $ \tau $ above is a skew pairing $ ( R \cmddotrtimes J ) \otimes ( S \cmddotrtimes K ) \rightarrow k $ such that $ \tau ( R ( n ) \otimes J , S ( m ) \otimes K ) = 0 $ if $ n \ne m $.
In this paper, given such $R$, $S$, $\tau$, we define the {\it generalized} $q$-{\it boson algebra} $ \cmdB$ (see Definition \ref{3.1}),
and prove for this, the result by Kashiwara \cite{K} and Nakashima \cite{N} cited above, under the assumptions that $ R ( 1 ) $, $ S ( 1 ) $ are finite-dimensional,
and $ \tau$ is non-degenerate, restricted to $ R ( 1 ) \otimes S ( 1 ) $.
We formulate the result as a category equivalence
\begin{equation*}
  \cmdO(\cmdB) \approx \mathop \cmdVec
\end{equation*}
between the integrable left $ \cmdB $-modules and the vector spaces; see Theorem \ref{3.13}.
The original result is recovered if we suppose that $ R \cmddotrtimes J = U_q^{\leq 0} $, $ S \cmddotrtimes K = U_q^{\geq 0} $,
and $ \tau$ is the Killing form given by Tanisaki \cite{Ts}.
It is a key for us to identify $ \cmdB $ with the generalized smash product $R \# S$ associated to $ \tau$,
which enables us to define a natural representation $ \rho : \cmdB = R \# S \rightarrow \cmdEnd ( R ) $ of $ \cmdB $ on $R$;
see Proposition \ref{3.4}.
It is not difficult to prove that $ \rho $ is injective, and its image is dense in $ \cmdEnd ( R ) $.
The latter immediately implies that $R$, regarded as a left $ \cmdB $-module (the Verma module) by $  \rho$,
is simple; see Theorem \ref{3.6}.
As its completion, $ \rho $ extends to an isomorphism $ \displaystyle \hat{\cmdB} \cmdarrow^{\simeq} \cmdEnd ( R ) $ of (complete) topological algebras.
We will define the {\it extremal projector} in $ \hat{ \cmdB } $ (see Definition \ref{3.15}),
which will play an important role to complete the proof of our main result.

Our main result cannot apply directly to Tan's situation, in which the matrix $ \cmdbbA $ may not be finite,
because $ R ( 1 ) $, $ S ( 1 ) $ then are not necessarily finite-dimensional.
But, we will refine Tan's result cited above, by slightly modifying the definition of $ \cmdO ( \cmdB )$ and the proof of our main result; see Theorem \ref{4.1}.
The main result will also apply to the special situation in which $ \cmdbbA $ is of finite type,
$q$ is a root of $1$, and $ U_q $ is replaced by a finite-dimensional quotient, $ u_q $, called the {\it Frobenius-Lusztig kernel}.
The associated $ \cmdB $ is then seen to be a finite-dimensional simple algebra which is isomorphic to $ \cmdEnd( R ) $ via $ \rho ( = \hat{\rho} ) $; see Section \ref{4.4}.

Preceding Sections 3, 4 whose contents were roughly described above,
Sections 1, 2 are devoted to showing preliminary results on braided Hopf algebras and skew pairings.

\section{Preliminaries on braided Hopf algebras}

Throughout this paper we work over a fixed field $k$, whose characteristic may be arbitrary except in the last Section 4.

\subsection{}
Let $J$ be a Hopf algebra.
We denote the coalgebra structure by
\begin{equation} \label{eq1.1}
  \Delta = \Delta_J : J \longrightarrow J \otimes J , \quad \Delta ( a ) = a_1 \otimes a_2 ; \quad \varepsilon : J \longrightarrow k,
\end{equation}
and the antipode by $ \cmdS = \cmdS_J$ in script.
We assume that $ \cmdS $ is bijective, and denote the composite-inverse by $ \overline{\cmdS} $, which will be called the {\it pode} of $J$.
Let $ {}_J^J \cmdYD $ denote the braided tensor category of Yetter-Drinfeld modules over $J$; see \cite[p.213]{Mo}.
Let $ V \in {}_J^J \cmdYD $.
Thus, $V$ is a left $J$-module, whose action will be denoted by $ a \cmdrhu v \ ( a \in J, v \in V ) $.
It is at the same time a left $J$-comodule, whose structure will be denoted by
\begin{equation} \label{eq1.2}
  \theta : V \longrightarrow J \otimes V, \quad \theta ( v ) = v_J \otimes v_V ,
\end{equation}
and satisfies
\begin{equation*}
  \theta (a \cmdrhu v) = a_1 v_J \cmdS ( a_3 ) \otimes ( a_2 \cmdrhu v_V ) \quad ( a \in J, v \in V ).
\end{equation*}
  The category ${}_J^J\cmdYD$ has the obvious tensor product and the braiding given by
\begin{equation} \label{eq1.3}
  c = c_{V , W} : V \otimes W \cmdarrow^{\simeq} W \otimes V, \quad c ( v \otimes w ) = (v_J \cmdrhu w) \otimes v_V,
\end{equation}
where $W$ is another object in ${}_J^J \cmdYD$.
The inverse of $c$ is given by
\begin{equation} \label{eq1.4}
  c^{-1} ( w \otimes v ) = v_V \otimes ( \overline{\cmdS} ( v_J ) \cmdrhu w).
\end{equation}

\subsection{}
Since ${}_J^J\cmdYD$ is braided, (graded) bialgebras or Hopf algebras in $ {}_J^J\cmdYD $ are defined in the natural manner,
which are called with the adjective `braided' added.
Let $V \in {}_J^J \cmdYD$, as above.
The tensor algebra $ T ( V ) $ of $V$ turns uniquely into a braided bialgebra in $ {}_J^J \cmdYD $,
if each element $ v \in V $ is supposed to be primitive so that $ \Delta ( v ) = v \otimes 1 + 1 \otimes v $.
In fact, $T ( V )$ is a braided graded Hopf algebra in $ {}_J^J \cmdYD $ with respect to the obvious grading.
Here and in what follows, gradings mean those by $ \cmdbbN = \{ 0, 1, 2, \cdots \} $ unless otherwise stated.
A {\it pre-Nichols algebra} \cite{M2} of $V$ is a quotient $ T ( V ) / I $ of $ T ( V ) $ by some homogeneous bi-ideal (necessarily, Hopf ideal) $I$ in $ {}_J^J \cmdYD $ such that $ I \cap V = 0 $.
In other words, it is a braided graded bialgebra $ R = \bigoplus_{n = 0}^{\infty} R ( n ) $ such that
\begin{eqnarray}
  & R ( 0 ) = k, \label{eq1.5} & \\
  & R ( 1 ) = V, \label{eq1.6} & \\
  & R \mbox{ is generated by } R ( 1 ), \label{eq1.7} &
\end{eqnarray}
The condition (\ref{eq1.5}) implies
\begin{equation*}
  ( V = ) R ( 1 ) \subset P ( R ) := \{ \mbox{all primitives in } R \},
\end{equation*}
and that $R$ is a braided (graded) Hopf algebra.
A pre-Nichols algebra $ T ( V ) / I = R $ of $V$ is said to be the {\it Nichols algebra} \cite{AS} of $V$,
if $I$ is the largest possible, or equivalently if $R ( 1 ) = P ( R )$.

\subsection{}
Let $ R = \bigoplus_{n = 0}^{\infty} R ( n )$ be such a braided graded Hopf algebra in $ {}_J^J \cmdYD $ that satisfies (\ref{eq1.5}).
To distinguish from the ordinary (or trivially braided) situation, we denote the coproduct of $R$ by
\begin{equation*}
  \Delta_R ( r ) = r_{ ( 1 ) } \otimes r_{ ( 2 ) } \quad (r \in R) ; \mbox{ cf. (\ref{eq1.1})}.
\end{equation*}
By Radford's biproduct construction (or bosonization), we have an ordinary Hopf algebra,
\begin{equation*}
  \cmdA = R \cmddotrtimes J.
\end{equation*}
This is denoted by $ R \times J $ in \cite{R}, and by $ R \# J $ in \cite{AS} and others; our notation is due to Shahn Majid.
As a vector space, $ \cmdA $ equals $ R \otimes J $; an element $ r \otimes a $ in $ \cmdA $ will be denoted simply by $ r a $.
$ \cmdA $ has the unit $ 1 \cdot 1 ( = 1 \otimes 1 ) $ and the counit $ \varepsilon \otimes \varepsilon $,
while its product and coproduct are given respectively by
\begin{eqnarray*}
  ( r a ) ( s b ) &=& r ( a_1 \cmdrhu s ) ( a_2 b ) \\
   \Delta ( r a ) &=& r_{( 1 )} ( r_{( 2 ) J} a_1 ) \otimes r_{( 2 ) R} a_2,
\end{eqnarray*}
where $r , s \in R, a , b \in J$.
Especially, if $v \in R ( 1 )$, then $v \in P ( R )$, whence
\begin{equation} \label{eq1.8}
  \Delta_{\cmdA} ( v ) = v \otimes 1 + v_J \otimes v_R.
\end{equation}
Notice that $ \cmdA$ is a graded Hopf algebra with $ \cmdA ( n ) = R ( n ) \otimes J$.
Since by (\ref{eq1.5}), the coradical of $ \cmdA$ is included in $J = \cmdA ( 0 )$, $ \cmdA$ has a bijective antipode.

Let $ \cmdA^{ \cmdcop }$ denote the algebra $ \cmdA$ which is given the coproduct $ \Delta^{ \cmdcop } ( \alpha ) = \alpha_2 \otimes \alpha_1 \ ( \alpha \in \cmdA )$ opposite to the original one.
By the original counit and grading, $ \cmdA^{\cmdcop}$ is a graded Hopf algebra with bijective antipode $ \overline{ \cmdS }_{ \cmdA } $.
The original Hopf algebra projection $ \varepsilon \otimes \cmdid : \cmdA = R \otimes J \rightarrow J$ gives a Hopf algebra projection $ \cmdA^{ \cmdcop } \rightarrow J^{ \cmdcop }$,
which we call $ \pi$.
Let
\begin{equation} \label{eq1.9}
  \overline{R} = \{ \alpha \in \cmdA^{ \cmdcop } \mid ( \cmdid \otimes \pi ) \circ \Delta^{\cmdcop} ( \alpha ) = \alpha \otimes 1 \}
\end{equation}
denote the (left coideal) subalgebra of $ \cmdA^{ \cmdcop }$ consisting of all right $J^{ \cmdcop }$-coinvariants along $ \pi$.
By Radford \cite{R}, $ \overline{R}$ forms a braided graded Hopf algebra in ${}_{ J^{\cmdcop} }^{ J^{\cmdcop} } \cmdYD$ such that $ \overline{R} \cmddotrtimes J^{ \cmdcop } = \cmdA^{ \cmdcop }$.

\begin{proposition} \label{1.1}
  Let $ \cmdS$, $ \overline{ \cmdS }$ denote the antipode and the pode of $ \cmdA$, respectively.
  \begin{enumerate}
    \item $ \overline{R} = \cmdS ( R ) = \overline{ \cmdS } ( R )$.
    \item $ \overline{R}$ is in ${}_{ J^{\cmdcop} }^{ J^{\cmdcop} } \cmdYD$ with respect to the structure given by
    \begin{eqnarray*}
      & & a^{ \cmdcop } \cmdrhu \overline{ \cmdS } ( r ) := \overline{ \cmdS } ( a \cmdrhu r ),\\
      & & \overline{\cmdS} ( r ) \mapsto \overline{ \cmdS } ( r_J ) \otimes \overline{ \cmdS } ( r_R ),
    \end{eqnarray*}
    where $r \in R$, and $a \in J$ with copy $a^{ \cmdcop } \in J^{ \cmdcop }$.
    \item The coproduct of $ \overline{R}$ is given by
    \begin{equation*}
      \Delta_{ \overline{R} } ( \overline{ \cmdS } ( r ) ) = \overline{ \cmdS } ( \overline{ \cmdS } ( r_{ ( 2 ) J } ) \cmdrhu r_{ ( 1 ) } ) \otimes \overline{ \cmdS } ( r_{ ( 2 ) R } ).
    \end{equation*}
    where $r \in R$.
  \end{enumerate}
\end{proposition}
\begin{proof}
  It is easy to prove (1).
  The remaining (2), (3) follow from \cite[Theorem3]{R}.
  To see (3), notice that the coproduct $ \Delta_{\cmdA} ( \overline{ \cmdS } ( r ) )$ of $ \cmdA$ is given by
  \begin{equation} \label{eq1.10}
      \Delta_{ \cmdA } ( \overline{ \cmdS } ( r ) ) = \overline{ \cmdS } ( r_{ (2) R } ) \otimes \overline{ \cmdS } ( \overline{ \cmdS } ( r_{ ( 2 ) J 1 } ) \cmdrhu r_{ ( 1 ) } ) \overline{ \cmdS } ( r_{ ( 2 ) J 2 } ).
  \end{equation}
\end{proof}

\begin{proposition} \label{1.2}
  With the same notation as above, set $V = R ( 1 )$ in $ {}_J^J \cmdYD $, and $ \overline{V} = \overline{R} ( 1 )$ in $ {}_{ J^{\cmdcop} }^{ J^{\cmdcop} } \cmdYD $.
  \begin{enumerate}
    \item $ \overline{V} = \cmdS ( V ) = \overline{ \cmdS } ( V )$.
    \item $R$ is a pre-Nichols algebra (resp., the Nichols algebra) of $V$ if and only if $ \overline{R}$ is a pre-Nichols algebra (resp., the Nichols algebra) of $ \overline{V}$.
  \end{enumerate}
\end{proposition}
\begin{proof}
  Part (1) and the assertion on `pre-Nichols' in (2) follow by Proposition \ref{1.1}, while the assertion on `Nichols' follows since we see directly that $ \overline{\cmdS}$ induces an isomorphism
  \begin{equation} \label{eq1.11}
    P ( R ) \cmdarrow^{\simeq} P ( \overline{R} ). 
  \end{equation}
  See also the last part of the following subsection.
\end{proof}

\subsection{}
Let $ V \in {}_J^J \cmdYD $.
Since $ J^{\cmdcop} = J $ as an algebra, $V$ may be regarded as a left $ J^{\cmdcop} $-module, which we denote by $ V^t $.
Then, $ V^t $ turns into an object in $ {}_{ J^{\cmdcop} }^{ J^{\cmdcop} } \cmdYD $ with respect to the modified comodule structure $ v \mapsto \overline{ \cmdS } ( v_J ) \otimes v_V$.
Moreover, $ V \mapsto V^t $ gives a category isomorphism $ \displaystyle {}_J^J \cmdYD \cmdarrow^{\simeq} {}_{ J^{\cmdcop} }^{ J^{\cmdcop} } \cmdYD $, which is involutive in the sense $ V^{tt} = V $.
This is a variation of (and is in fact essentially the same as) the isomorphism $ \displaystyle {}_J^J \cmdYD \cmdarrow^{\simeq} {}_{ J^{\cmdop} }^{ J^{\cmdop} } \cmdYD $ given by Radford and Schneider \cite[Section 2]{RS}.

Suppose that $ R = ( R , m , \Delta ) $ is a braided bialgebra in $ {}_J^J \cmdYD $ with product $m$ and coproduct $ \Delta $.
Give on $ R^t $ the opposite product $ m^t $ defined by
\begin{equation*}
  m^t ( r \otimes s ) := m ( s \otimes r ) \quad ( r , s \in R ),
\end{equation*}
and the coproduct $ \Delta^t$ defined by
\begin{equation*}
  \Delta^t := \cmdtw \circ c_{R , R}^{-1} \circ \Delta,
\end{equation*}
where $c_{R , R}^{-1}$ is such as given by (\ref{eq1.4}), and $ \cmdtw$ denotes the twist map $r \otimes s \mapsto s \otimes r$.
Explicitly,
\begin{equation}
\label{eq1.12}
  \Delta^t ( r ) = ( \overline{ \cmdS } ( r_{( 2 ) J} ) \cmdrhu r_{( 1 )} ) \otimes r_{( 2 ) R} \quad ( r \in R ).
\end{equation}

\begin{proposition} \label{1.3}
  \begin{enumerate}
    \item Given the same unit and counit as the original ones, $ R^t = ( R^t , \Delta^t , m^t ) $ is a braided bialgebra in $ {}_{ J^{\cmdcop} }^{ J^{\cmdcop} } \cmdYD $.
    \item We have $ R^{tt} = R $, \ $ P ( R ) = P ( R^t ) $.
  \end{enumerate}
\end{proposition}
\begin{proof}
  (1) As is explained in \cite[Section 2]{RS}, the category isomorphism \\
  $ \displaystyle F : {}_J^J \cmdYD \cmdarrow^{\simeq} {}_{ J^{\cmdcop} }^{ J^{\cmdcop} } \cmdYD $,
  $ F ( V ) = V^t $ is tensorial with respect to $ \displaystyle \cmdid : k \rightarrow k = F ( k ) $, $\displaystyle c_{W , V} \circ \cmdtw : F ( V ) \otimes F ( W ) \cmdarrow^{\simeq} F ( V \otimes W ) $.
  Therefore, $F ( R )$ is a braided bialgebra in $ {}_{ J^{\cmdcop} }^{ J^{\cmdcop} } \cmdYD $.
  Our $R^t$ is the braided opposite to $F ( R )$.
  
  (2) This is directly verified.
\end{proof}

Obviously, this result is generalized to braided graded bialgebras.

\begin{proposition} \label{1.4}
  Suppose that we are in the same situation as in Propositions \ref{eq1.1}, \ref{eq1.2}.
  \begin{enumerate}
    \item $R$ is a pre-Nichols algebra (resp., the Nichols algebra) of $V$ if and only if $R^t$ is a pre-Nichols algebra (resp., the Nichols algebra) of $V^t$.
    \item The pode $ \overline{\cmdS}$ of $R \cmddotrtimes J$ gives isomorphisms
    \begin{equation*}
      R \cmdarrow^{\simeq} \overline{R}^t , \quad R^t \cmdarrow^{\simeq} \overline{R}
    \end{equation*}
    of braided graded Hopf algebras in ${}_J^J \cmdYD$,
    and in ${}_{ J^{\cmdcop} }^{ J^{\cmdcop} } \cmdYD$, respectively.
  \end{enumerate}
\end{proposition}
\begin{proof}
  (1) This follows from Proposition \ref{1.3} (2).
  
  (2) This follows from Proposition \ref{1.1} (2), (3) and (\ref{eq1.12}).
\end{proof}

Part(2) above refines the isomorphism given in (\ref{eq1.11}), more conceptually.

\subsection{}
For a vector space $X$, we let $X^* = \cmdHom ( X , k )$ denote the linear dual,
and $ \langle \ , \ \rangle : X^* \times X \rightarrow k$ the evaluation map.

Let $J^{\circ}$ denote the dual Hopf algebra \cite[Section 6.2]{Sw} of $J$, which also has a bijective antipode.
If $V \in {}_J^J \cmdYD$ is finite-dimensional,
$V^*$ is in ${}_{ J^{\circ} }^{ J^{\circ} } \cmdYD$ with respect to the structure $p \cmdrhu v^* $, $v^* \mapsto v_{J^{\circ}}^{*} \otimes v_{V^*}^{*}$ determined by
\begin{eqnarray*}
  & \langle p \cmdrhu v^* , v \rangle = \langle p , v_J \rangle \langle v^* , v_V \rangle , & \\
  & \langle v_{ J^{\circ} }^{*} , a \rangle \langle v_{V^*}^{*} , v \rangle = \langle v^* , a \cmdrhu v \rangle , &
\end{eqnarray*}
where $v \in V $, $ v^* \in V^* $, $ a \in J $, $ p \in J^{\circ} $,
Moreover, $V \mapsto V^*$ gives a braided tensor contravariant-functor from the finite-dimensional objects in ${}_J^J \cmdYD$ to those objects in ${}_{ J^{\circ} }^{ J^{\circ} } \cmdYD$,
where the tensorial structure is given by the canonical isomorphisms $ \displaystyle k \cmdarrow^{\simeq} k^* $, $ \displaystyle V^* \otimes W^* \cmdarrow^{\simeq} ( V \otimes W )^{*}$.

Let $ R = \bigoplus_{n = 0}^{\infty} R ( n )$ be a braided graded Hopf algebra in ${}_J^J \cmdYD$ such that $R ( 0 ) = k$.
Following \cite[Section 11.2]{Sw}, we define the {\it graded dual} $R^g$ of $R$ by
\begin{equation*}
  R^g = \bigoplus_{n = 0}^{\infty} R ( n )^{*}.
\end{equation*}
This is a subalgebra of the dual algebra $R^*$ of the coalgebra $R$.

Assume that $R$ is {\it locally finite} \cite[ibidem]{Sw} in the sense that each component $ R ( n ) $ is finite-dimensional.
Since each $ R ( n )^{*} $ is then an object in $ {}_{ J^{\circ} }^{ J^{\circ} } \cmdYD $, $ R^g $ is, too.
We see now easily the following.

\begin{lemma} \label{1.5}
  Under the assumption above,
  $R^g$ is a braided graded Hopf algebra in $ {}_{ J^{\circ} }^{ J^{\circ} } \cmdYD $ such that $ R^g ( 0 ) = k $.
\end{lemma}

\subsection{}
Given distinct Hopf algebras $J$, $K$ with bijective antipode, suppose $ V \in {}_J^J \cmdYD $, $ W \in {}_K^K \cmdYD $.
A linear map $ \psi : V \rightarrow W $ is called a {\it map of braided vector spaces} if there is a Hopf algebra map $ \phi : J \rightarrow K$ with which $ \psi $ is $ \phi $-linear and colinear \cite{RS} in the sense that
\begin{eqnarray}
  \psi ( a \cmdrhu v ) &=& \phi ( a ) \cmdrhu \psi ( v ), \label{eq1.13} \\
  \phi ( v_J ) \otimes \psi ( v_V ) &=& \psi ( v )_K \otimes \psi ( v )_W, \label{eq1.14}
\end{eqnarray}
where $ a \in J $, $ v \in V $.
In this case, to specify $ \phi $, we will say that $ \psi $ is {\it attended by} $ \phi$.
Notice that $ \psi$ then preserves the braiding in the obvious sense.
The definition above extends in the obvious way to braided (graded) bi- or Hopf algebras.
For example in the situation of Lemma \ref{1.5},
the canonical isomorphism $ \displaystyle R \cmdarrow^{\simeq} ( R^g )^g $ is a map of braided graded Hopf algebras,
attended by the canonical Hopf algebra map $ J \rightarrow ( J^{\circ} )^{\circ} $.

\section{Skew pairings on pre-Nichols algebras}

\subsection{}

Let $A$, $X$ be vector spaces.
Given a linear form $ \tau : A \otimes X \rightarrow k $, we denote the adjoint linear maps by
\begin{eqnarray*}
  \tau^l &:& A \rightarrow X^* , \quad \tau^l ( a ) ( x ) = \tau ( a , x ) ,\\
  \tau^r &:& X \rightarrow A^* , \quad \tau^r ( x ) ( a ) = \tau ( a , x ) ,
\end{eqnarray*}
where $ a \in A $, $ x \in X $.
For subspaces $ B \subset A $, $ Y \subset X $, we let
\begin{equation*}
  \tau_{ B Y } := \tau \mid _{ B \otimes Y } : B \otimes Y \rightarrow k
\end{equation*}
denote the restriction.

Suppose that $A$, $X$ are Hopf algebras with bijective antipode.
A linear form $ \tau : A \otimes X \rightarrow k $ is called a {\it skew pairing} \cite[Definition 1.3]{DT}, if
\begin{eqnarray*}
  & \tau ( a b , x ) = \tau ( a , x_1 ) \tau ( b , x_2 ) , & \\
  & \tau ( a , x y ) = \tau ( a_1 , y ) \tau ( a_2 , x ) , & \\
  & \tau ( 1 , x ) = \varepsilon ( x ) , \quad \tau ( a , 1 ) = \varepsilon ( a ) &
\end{eqnarray*}
for all $ a , b \in A , x , y \in X $, or equivalently if $ \tau^l $ gives a Hopf algebra map $ A^{\cmdcop} \rightarrow X^{\circ} $,
or equivalently if $ \tau^r $ gives a Hopf algebra map $ X \rightarrow ( A^{\cmdcop} ) ^{\circ} $.
A skew pairing $ \tau $ necessarily has a convolution-inverse $ \tau^{-1} $, such that
\begin{equation*}
  \tau^{-1} ( a , x ) = \tau ( \overline{ \cmdS } ( a ) , x ) = \tau ( a , \cmdS ( x ) ) \quad ( a \in A , x \in X ).
\end{equation*}

\subsection{}

In what follows until the end of Section 3, We will work in the following situation.
First, let $J$, $K$ be Hopf algebras with bijective antipode,
and suppose that a skew pairing $ \tau_{0} : J \otimes K \rightarrow k $ is given.
Next, let $ V \in {}_J^J \cmdYD $, $ W \in {}_K^K \cmdYD $, and suppose that $ \tau_1 : V \otimes W \rightarrow k $ is a linear form such that
\begin{eqnarray} 
  \tau_1 ( a \cmdrhu v , w ) &=& \tau_0 ( a , w_K ) \tau_1 ( v , w_W ), \label{eq2.1} \\
  \tau_1 ( v , x \cmdrhu w ) &=& \tau_0 ( v_J , \cmdS ( x ) ) \tau_1 ( v_V , w ) \label{eq2.2}
\end{eqnarray}
for all $ a \in J $, $ x \in K $, $ v \in V $, $ w \in W $.
If $W$ is finite-dimensional, these last conditions are equivalent to that $ \tau_1^l $ gives a map $ V^t \rightarrow W^* $ of braided vector spaces,
attended by $ \tau_0^l : J^{ \cmdcop } \rightarrow K^{ \circ } $.
Similarly we have another equivalent condition if $V$ is finite-dimensional; see \cite[Proposition 4.2]{RS}.
Finally, let $R$, $S$ be pre-Nichols algebras of $V$, $W$, respectively.

\begin{proposition} \label{2.1}
  $ \tau_0 $, $ \tau_1 $ extend uniquely to a skew pairing $ \tau : ( R \cmddotrtimes J ) \otimes ( S \cmddotrtimes K ) \rightarrow k $ such that
  \begin{equation} \label{eq2.3}
    \tau ( a , w ) = 0= \tau ( v , x )
  \end{equation}
  for all $ a \in J $, $ x \in K $, $ v \in V $, $ w \in W $.
\end{proposition}
\begin{proof}
  This follows by \cite[Theorem 5.3]{M2} if we take our $ \tau_1 ( v , w ) $ as $ - \lambda ( w , v ) $ in \cite{M2}.
  See also \cite[Theorem 8.3]{RS}.
\end{proof}

Conversely, if $ \tau $ is a skew pairing which extends $ \tau_0 $, and satisfies (\ref{eq2.3}),
it restricts to the linear form $ \tau_{ V W } $ which satisfies (\ref{eq2.1}) (\ref{eq2.2}), as is easily seen.
In what follows we keep to denote by $ \tau $ the skew pairing as above.

\begin{lemma} \label{2.2}
  If $ a \in K $, $ x \in K $, $ r , r' \in R $, $ s , s' \in S $,
  then we have
  \begin{eqnarray} 
    \tau ( r a , s x ) &=& \tau ( r_J , x_1 ) \tau ( r_R , s ) \tau ( a , x_2 ) , \label{eq2.4} \\
    \tau ( r , s s' ) &=& \tau ( r_{( 1 )} , s' ) \tau ( r_{( 2 )} , s ) , \label{eq2.5} \\
    \tau ( r r' , s ) &=& \tau ( r , \overline{ \cmdS } ( s_{(2) K} ) \cmdrhu s_{( 1 )} ) \tau ( r' , s_{( 2 ) S} ) , \label{eq2.6}
  \end{eqnarray}
\end{lemma}
\begin{proof}
  We reach these formulae by proving the following, step by step.
  \begin{eqnarray}
    & \tau ( a , s ) = \varepsilon ( a ) \varepsilon ( s ) , \quad \tau ( r , x ) = \varepsilon ( r ) \varepsilon ( x ), & \nonumber \\
    & \tau ( r a , s ) = \varepsilon ( a ) \tau ( r , s ) , \label{eq2.7} & \\
    & \tau ( r , s x ) = \tau ( r_J , x ) \tau ( r_R , s ) , \label{eq2.8} &
  \end{eqnarray}
  To prove (\ref{eq2.6}), notice from (\ref{eq2.8}) that
  \begin{equation*}
    \tau ( r , x \cmdrhu s ) = \tau ( r_J , \cmdS ( x ) ) \tau ( r_R , s ).
  \end{equation*}
\end{proof}

\begin{lemma} \label{2.3}
  Let $ \overline{\cmdS} $ denote the pode of $ R \cmddotrtimes J $.
  If $ a \in J $, $ x \in K $, $ r , r' \in R $, $ s , s' \in S $,
  then we have
  \begin{eqnarray}
    \tau ( \overline{ \cmdS } ( r ) a , s x ) &=& \tau ( \overline{ \cmdS } ( r ) , s ) \tau ( a , x ) , \label{eq2.9} \\
    \tau ( \overline{ \cmdS } ( r ) \overline{ \cmdS } ( r' ) , s ) &=& \tau ( \overline{ \cmdS } ( r ) , s_{( 1 )} ) \tau ( \overline{ \cmdS } ( r' ) , s_{( 2 )} ) ,  \label{eq2.10} \\
    \tau ( \overline{ \cmdS } ( r ) , s s' ) &=& \tau ( \overline{ \cmdS } ( \overline{ \cmdS } ( r_{( 2 ) J } ) \cmdrhu r_{( 1 )} ) , s ) \tau ( \overline{ \cmdS } ( r_{( 2 ) R} ) , s' ).  \label{eq2.11}
  \end{eqnarray}
\end{lemma}
\begin{proof}
  Suppose $ v \in V $. Then, $ \overline{ \cmdS } ( v ) = - v_V \overline{ \cmdS } ( v_J ) $.
  By (\ref{eq2.8}), we see
  \begin{equation} \label{eq2.12}
    \tau ( \overline{ \cmdS } ( v ) , s x ) = - \varepsilon ( x ) \tau ( v , x ) ,
  \end{equation}
  It follows by induction on $n$ that if $ v^1 , \cdots , v^n \in V $, then
  \begin{equation} \label{eq2.13}
    \tau ( \overline{ \cmdS } ( v^1 ) \cdots \overline{ \cmdS } ( v^n ) , s x ) = ( - 1 )^n \varepsilon ( x ) \tau ( v^1 , s_{( 1 )} ) \cdots \tau ( v^n , s_{( n )}) .
  \end{equation}
  This implies (\ref{eq2.9}), (\ref{eq2.10}); recall here (\ref{eq1.7}).
  We see that (\ref{eq2.11}) follows from (\ref{eq1.10}), (\ref{eq2.9}).
\end{proof}

Recall from \ref{eq1.3} that $ \overline{R} = \overline{ \cmdS } ( R ) $ in $ R \cmddotrtimes J $.

\begin{corollary} \label{2.4}
  \begin{enumerate}
    \item $ \tau $ is non-degenerate if and only if $ \tau_{ \overline{R} S } $ and $ \tau_0 ( = \tau_{J K} ) $ are both non-degenerate.
    \item $ \tau ( R ( n ) \otimes J , S ( m ) \otimes K ) = 0 $ if $ n \ne m $.
  \end{enumerate}
\end{corollary}
\begin{proof}
  (1) This follows by (\ref{eq2.9}).
  
  (2) One sees from (\ref{eq2.13}) that
  \begin{equation}
    \tau ( \overline{R} ( n ) , S ( m ) ) = 0 \quad \mbox{ if } \quad n \ne m.
  \end{equation}
  Since $ \overline{R} ( n ) \otimes J = R ( n ) \otimes J $ for each $n$,
  Part (2) follows by (\ref{eq2.9}), again.
\end{proof}

By Part (2) above, graded linear (at least) maps
\begin{equation*}
  R \rightarrow S^g , \quad S \rightarrow R^g , \quad \overline{R} \rightarrow S^g , \quad S \rightarrow \overline{R}^{g}
\end{equation*}
are given by $ \tau_{R S}^{l} $, $ \tau_{R S}^{r} $, $ \tau_{\overline{R} S}^{l} $, $ \tau_{\overline{R} S}^{r} $, respectively.

\subsection{}
Set $ \overline{V} = \overline{R} ( 1 ) $.
Notice from Proposition \ref{1.2} that $ \overline{R} $ is a pre-Nichols algebra of $ \overline{V} $ in $ {}_{ J^{\cmdcop} }^{ J^{\cmdcop} } \cmdYD $,
such that $ \overline{R} \cmddotrtimes J^{\cmdcop} = ( R \cmddotrtimes J )^{\cmdcop} $.
Recall from \ref{eq1.4} that $ S^t $ is a pre-Nichols algebra of $ W^t $ in $ {}_{ K^{\cmdcop} }^{ K^{\cmdcop} } \cmdYD $.

\begin{lemma} \label{2.5}
  Assume that $V$, $W$ are both finite-dimensional so that $R$, $S$ are locally finite.
  \begin{enumerate}
    \item $ \tau_{R S}^{l} $, $ \tau_{R S}^{r} $ give maps
    \begin{equation} \label{eq2.15}
      R \rightarrow ( S^t )^g , \quad S^t \rightarrow R^g
    \end{equation}
    of braided graded Hopf algebras, which are graded dual to each other.
    \item $ \tau_{\overline{R} S}^{l} $, $ \tau_{\overline{R} S}^{r} $ give maps
    \begin{equation} \label{eq2.16}
      \overline{R} \rightarrow S^g , \quad S \rightarrow \overline{R}^g
    \end{equation}
    of braided graded Hopf algebras, which are graded dual to each other.
  \end{enumerate}
\end{lemma}
\begin{proof}
  (1) This follows from (\ref{eq2.5}), (\ref{eq2.6}).
  Notice that the map $ R \rightarrow ( S^t )^g $ above is attended by $ \cmdS_{K^{\circ}} \circ \tau_{0}^{l} : J \rightarrow {(K^{\cmdcop})}^{\circ} $,
  since it is so, restricted to $V$, and $R$ is generated by $V$.
  
  (2) Similarly, this follows from (\ref{eq2.10}), (\ref{eq2.11}); see also Proposition \ref{1.1} (3).
\end{proof}

\begin{proposition} \label{2.6}
  \begin{enumerate}
    \item The following are equivalent:
    \begin{itemize}
      \item[(a)] $ \tau_{R S} $ is non-degenerate;
      \item[(b)] $ \tau_1 ( = \tau_{V W} ) $ is non-degenerate, and $R$, $S$ are both Nichols algebras;
      \item[(c)] $ \tau_{\overline{R} S} $ is non-degenerate;
      \item[(d)] $ \tau_{\overline{V} W} $ is non-degenerate, and $ \overline{R} $, $S$ are both Nichols algebras.
    \end{itemize}
    These conditions hold true if $\tau$ is non-degenerate.
    \item Assume that $V$, $W$ are both finite-dimensional.
    Then the equivalent conditions given above are further equivalent to each of the following:
    \begin{itemize}
      \item[(e)] The maps given in (\ref{eq2.15}) are isomorphisms;
      \item[(f)] The maps given in (\ref{eq2.16}) are isomorphisms.
    \end{itemize}
  \end{enumerate}
\end{proposition}
\begin{proof}
  (1) Proposition \ref{1.2} (2) and (\ref{eq2.12}) prove (b) $ \Leftrightarrow $ (d).
  To prove (a) $ \Leftrightarrow $ (b), let $ T = \cmdIm \tau_{R S}^{l} $ denote the image of $ \tau_{R S}^{l} $.
  We see from (\ref{eq2.5}) that $T$ is a subcoalgebra of the dual coalgebra $ S^{\circ} $ of the algebra $S$; see \cite[Section 6.0]{Sw}.
  Moreover, $T$ is a graded coalgebra, and $ \tau_{RS}^l : R \rightarrow T $ is a graded anti-coalgebra map.
  Since $S$ is generated by $ W = S ( 1 ) $, $T$ is strictly graded, or $ P ( T ) = T ( 1 ) $; see \cite[Section 11.2]{Sw}.
  It follows that $ \tau_{R S}^{l} $ is injective if and only if $ \tau_{1}^{l} $ is injective, and $R$ is Nichols.
  Similarly, since by (\ref{eq2.6}), $ \tau_{R S}^{r} : S^t \rightarrow \cmdIm \tau_{R S}^{r} $ is a graded coalgebra map onto a strictly graded coalgebra,
  it follows by using Proposition \ref{1.4} (1) that $ \tau_{R S}^{r} $ is injective if and only if $ \tau_{1}^{r} $ is injective, and $S$ is Nichols.
  These prove the desired equivalence.
  Similarly, (c) $ \Leftrightarrow $ (d) follows.
  The last statement follows by Corollary \ref{2.4} (1).
  
  (2) This follows, since obviously, (a) $ \Leftrightarrow $ (e), (c) $ \Leftrightarrow $ (f).
\end{proof}

\begin{remark} \label{2.7}
  As was seen above, the graded linear map $ R \rightarrow S^g $ given by $ \tau_{R S}^{l} $ is injective if and only if $ \tau_{1}^{l} $ is injective, and $R$ is Nichols.
\end{remark}

\section{Generalized $q$-boson algebras and their integrable modules}

Throughout this section we let
\begin{equation} \label{eq3.1}
  J , \ K , \ \tau_0 , \ V , \ W , \ \tau_1 , \ R , \ S
\end{equation}
be as given at the beginning of \ref{2.2}

\subsection{}

>From the proof of \cite[Theorem 5.3]{M2}, we see that the tensor-product algebra $ R \otimes S $ is uniquely deformed to such an algebra that includes $ R ( = R \otimes k ) $, $ S ( = k \otimes S ) $ as subalgebras, and obeys the rules of product
\begin{eqnarray}
  & r s = r \otimes s \quad ( r \in R , s \in S ), & \label{eq3.2} \\
  & w v = \tau_0 ( v_J , w_J ) v_V \otimes w_W + \tau_1 ( v , w ) 1 \otimes 1 \quad ( v \in V , w \in W ). & \label{eq3.3} 
\end{eqnarray}
Recall here that we take our $ \tau_1 ( v , w ) $ as $ - \lambda ( w , v ) $ in \cite{M2}; see also Remark \ref{3.3} below.

\begin{definition} \label{3.1}
  We denote this deformed algebra by $ \cmdB $, and call it the {\rm generalized} $q$-{\rm boson algebra} associated to the data (\ref{eq3.1}).
\end{definition}

Notice that $ \cmdB $ has $ 1 \otimes 1 $ as unit.
Also, it is a $ \cmdbbZ $-graded algebra, by counting degrees so that
\begin{equation*}
  \cmddeg R ( n ) = n , \quad \cmddeg S ( n ) = - n \quad ( n \geq 0 ) 
\end{equation*}

\subsection{}
To obtain a useful description of $ \cmdB $,
let $ \tau $ be the skew pairing given by Proposition \ref{2.1}.
We set
\begin{equation*}
  \cmdA = R \cmddotrtimes J , \quad \cmdH = S \cmddotrtimes K ,
\end{equation*}
so that $ \tau $ is defined on $ \cmdA \otimes \cmdH $.

\begin{lemma}[cf. \cite{Lu}] \label{3.2}
  \begin{enumerate}
    \item $ \cmdA $ is a left $ \cmdH $-module algebra under the action
    \begin{equation*}
      \xi \triangleright \alpha = \tau ( \alpha_1 , \xi ) \alpha_2 \quad ( \alpha \in \cmdA , \ \xi \in \cmdH ).
    \end{equation*}
    \item The associated smash product $ \cmdA \# \cmdH $ is the algebra defined on $ \cmdA \otimes \cmdH $ with respect to the unit $ 1 \# 1 $ and the product
    \begin{equation*}
      ( \alpha \# \xi ) ( \beta \# \eta ) = \tau ( \beta_1 , \xi_1 ) \alpha \beta_2 \# \xi_2 \eta,
    \end{equation*}
    where $ \alpha , \beta \in \cmdA $, $ \xi , \eta \in \cmdH $.
    Here, $ \alpha \# \xi $ stands for $ \alpha \otimes \xi $.
    \item A linear representation $ \varpi : \cmdA \# \cmdH \rightarrow \cmdEnd ( \cmdA ) $ of $ \cmdA \# \cmdH $ on $ \cmdA $ is given by
    \begin{equation*}
      \varpi ( \alpha \# \xi ) ( \beta ) = \alpha ( \xi \triangleright \beta ) \quad ( \alpha , \beta \in \cmdA , \ \xi \in \cmdH ).
    \end{equation*}
  \end{enumerate}
\end{lemma}

This is easy to see.
Lu \cite{Lu} essentially proved the result above, when $ \cmdH $ is a finite-dimensional Hopf algebra,
$ \cmdA = ( \cmdH^{ \cmdcop } )^{*} $, and $ \tau $ is the evaluation map;
in this case, $ \varpi $ is necessarily an isomorphism.

\begin{remark} \label{3.3}
  With the notation as above,
  $ \sigma ( \alpha \otimes \xi , \beta \otimes \eta ) := \varepsilon ( \alpha ) \tau ( \beta , \xi ) \varepsilon ( \eta ) $ defines a $2$-cocyle on $ ( \cmdA \otimes \cmdH )^{ \otimes 2 } $;
  see \cite[Proposition 2.6]{M2}, for example.
  We see that the algebra $ \cmdA \# \cmdH $ coincides with the crossed product $ {}_{ \sigma } ( \cmdA \otimes \cmdH ) $ associated to $ \sigma $,
  and with the bicleft (especially, right $ \cmdA \otimes \cmdH $-cleft) extension over $k$ as given in the last paragraph of the proof of \cite[Theorem 5.3]{M2}.
  To prove the existance of such a skew pairing $ \tau $ as given by Proposition \ref{2.1},
  this last article first constructed that cleft extension,
  and then proved that the associated $2$-cocyle arises precisely from the desired $ \tau $.
  
\end{remark}


Since $R$, $S$ are left coideal subalgebras of $ \cmdA = R \cmddotrtimes J $, $ \cmdH = S \cmddotrtimes K $, respectively,
we see that $ R \otimes S ( = ( R \otimes k ) \otimes ( S \otimes k ) ) $ is a subalgebra of $ \cmdA \# \cmdH $.
We denote the resulting algebra by $ R \# S $.

\begin{proposition} \label{3.4}
  \begin{enumerate}
    \item The inclusions $ R , S \hookrightarrow R \# S $ uniquely extend to an algebra isomorphism $ \displaystyle \cmdB \cmdarrow^{\simeq} R \# S $.
    \item A linear representation $ \rho : R \# S \rightarrow \cmdEnd ( R ) $ of $ R \# S $ on $R$ is induced from $ \varpi $ so that
    \begin{equation*}
      \rho ( r \# s ) ( r' ) = r ( s \triangleright r' ) \quad ( r , r' \in R , \ s \in S ).
    \end{equation*}
  \end{enumerate}
\end{proposition}
\begin{proof}
  (1) We see from (\ref{eq1.8}) that the same relation as (\ref{eq3.3}) holds true in $ R \# S $.
  Therefore the inclusions above uniquely extend to an algebra map $ \cmdB \rightarrow R \# S $,
  which is an isomorphism since it is the identity map of $ R \otimes S $, regarded as a linear map.

  (2) This follows since $ r ( s \triangleright r' ) \in R $.
\end{proof}

\begin{remark} \label{3.5}
  By (\ref{eq2.7}), one sees
  \begin{equation} \label{eq3.4}
    s \triangleright r = \tau ( r_{( 1 )} , s ) r_{( 2 )} \quad ( r \in R , \ s \in S ).
  \end{equation}
  Let $ w \in W ( = S ( 1 ) ) $.
  One then sees
  \begin{equation} \label{eq3.5}
    w \triangleright r r' = ( w \triangleright r ) r' + \tau ( r_J , w_K ) r_R ( w_S \triangleright r' ) \quad ( r , r' \in R ).
  \end{equation}
  Thus, $w$ acts on $R$ as a sort of (braided) derivation.
\end{remark}

\subsection{}

We will identify so as $ \cmdB = R \# S $ via the isomorphism above.
We regard $R$ as a left $ \cmdB $-module by $ \rho $.
Notice that this $ \cmdB $-module is $ \cmdbbZ $-graded, or more explicitly
\begin{equation} \label{eq3.6}
  R ( l ) ( S ( m ) \triangleright R ( n ) ) \subset R ( l - m + n ) \quad ( l, m , n \geq 0 ),
\end{equation}
where we suppose $ R ( n ) = 0 $ if $ n < 0 $. Set $ \cmdI = R \otimes ( \bigoplus_{n > 0} S ( n ) ) $ in $ \cmdB $;
this is the left ideal generated by $ S ( 1 ) $.
We see that $ \beta \mapsto \rho ( \beta ) ( 1 ) \ ( \beta \in \cmdB ) $ induces the $ \cmdB $-linear isomorphism
\begin{equation} \label{eq3.7}
  \cmdB / \cmdI \cmdarrow^{\simeq} R .
\end{equation}

\begin{theorem} \label{3.6}
  Assume that $ \tau_{1}^{l} : V \rightarrow W^* $ is injective, and $R$ is the Nichols algebra of $V$.
  Then, $R$ is simple as a left $ \cmdB $-module.
\end{theorem}

To prove the theorem, let first
\begin{equation*}
  \nu : R \otimes R^g \rightarrowtail \cmdEnd(R)
\end{equation*}
denote the canonical linear injection given by
\begin{equation*}
  \nu ( r \otimes f ) ( r' ) = r \langle f , r' \rangle \quad ( r , r' \in R , \ f \in R^g ) .
\end{equation*}
Since $R$ is a braided Hopf algebra, we have the right $R$-linear isomorphism
\begin{equation*}
  \chi : R \otimes R \cmdarrow^{\simeq} R \otimes R , \quad \chi ( r \otimes r' ) = r_{( 1 )} \otimes r_{( 2 )} r',
\end{equation*}
just as in the case of ordinary Hopf algebras.
The right $R$-linear dual induces a linear isomorphism,
\begin{equation*}
  \chi^{\vee} : \cmdEnd ( R ) \cmdarrow^{\simeq} \cmdEnd ( R ) .
\end{equation*}

\begin{lemma} \label{3.7}
  Let $ \mu $ denote the composite $ \chi^{\vee} \circ \nu $.
  \begin{enumerate}
    \item $ \mu : R \otimes R^g \rightarrow \cmdEnd ( R ) $ is the linear injection given by
    \begin{equation*}
      \mu ( r \otimes f ) ( r' ) = r \langle f , r'_{(1)} \rangle r'_{( 2 )} \quad ( r , r' \in R , \ f \in R^g ).
    \end{equation*}
    \item The composite
    \begin{equation*}
      \cmdB = R \otimes S \cmdarrow^{ \cmdid \otimes \tau_{R S}^r } R \otimes R^g \cmdarrow^{ \mu } \cmdEnd ( R )
    \end{equation*}
    coincides with $ \rho $.
  \end{enumerate}
\end{lemma}

This is easy to prove.

Here, recall that a vector space $X$ given a topology is called a {\it topological vector space} \cite{T},
if the translation $ y \mapsto x + y $ by each $ x \in X $ is continuous
and if $X$ has a basis of neighborhoods of $0$ consisting of (vector) subspaces.
The topological vector spaces which we will encounter in this paper are all Hausdorff.
An algebra is called a {\it topological algebra}, if it is a topological vector space and if the product is continuous.

Given a discrete vector space $X$,
present it as a directed union $ X = \bigcup_{\lambda} X_{\lambda} $ of finite-dimensional subspaces $ X_{\lambda} \subset X $.
Then we have the canonical isomorphism
\begin{equation*}
  \cmdEnd ( X ) \simeq \lim_{\displaystyle \mathop{\longleftarrow}_{\lambda}} \cmdHom ( X_{\lambda} , X ) .
\end{equation*}
Through this, regard $ \cmdEnd ( X ) $ as the projective limit of discrete $ \cmdHom ( X_{\lambda} , X ) $.
Then, $ \cmdEnd ( X ) $ turns into a complete topological algebra;
the thus introduced topology is independent of choice of presentations $ X = \bigcup_{\lambda} X_{\lambda} $.

\begin{proof}[Proof of Theorem \ref{3.6}]
  Regarding $ \cmdEnd ( R ) $ as a topological vector space, as above,
  we claim that the image of $ \rho $ is dense in $ \cmdEnd ( R ) $.
  Notice that $ \chi^{\vee} $ is an isomorphism of topological vector spaces since $ \chi ( C \otimes R ) = C \otimes R $ for every finite-dimensional subcoalgebra $ C \subset R $.
  It then suffices to prove that the image of the composite $ \nu_{\cmdB} := \nu \circ ( \cmdid \otimes \tau_{R S}^r ) $ is dense.
  By Remark \ref{2.7}, the assumptions imply that $ \tau_{R S}^l $ is injective.
  It follows that given a finite-dimensional subspace $ X \subset R $, the composite
  \begin{equation*}
    S \cmdarrow^{ \tau_{R S}^r } R^g \longrightarrow X^*
  \end{equation*}
  is surjective, whence so is the composite
  \begin{equation*}
    \cmdB = R \otimes S \cmdarrow^{ \nu_{ \cmdB } } \cmdEnd ( R ) \longrightarrow \cmdHom ( X , R ) ,
  \end{equation*}
  where the second arrows are the restriction maps.
  This proves the claim.
  Therefore, given a finite number of linearly independent elements $ r_1 , \cdots , r_n $ in $R$ and the same number of arbitrary elements $ r'_1 , \cdots , r'_n $ in $R$,
  there exists $ \beta $ in $ \cmdB $ such that $ \beta r_i = r'_i $ for all $ 1 \leq i \leq n $.
  This fact even in $ n = 1 $ proves that $R$ cannot include any non-trivial $ \cmdB $-submodule.
\end{proof}

\begin{definition} \label{3.8}
  Let $M$ be a left $ \cmdB $-module.
  We define a subspace $ M_0 $ of $M$ by
  \begin{equation*}
    M_0 = \{ m \in M \mid S ( 1 ) m = 0 \}.
  \end{equation*}
  We also define a linear map $ \kappa_M $ by
  \begin{equation} \label{eq3.8}
    \kappa_M : R \otimes M_0 \rightarrow M , \quad \kappa_M ( r \otimes m ) = r m .
  \end{equation}
\end{definition}

\begin{lemma} \label{3.9}
  Regard $ R \otimes M_0 $ as a left $ \cmdB $-module by letting $ \cmdB $ act on the factor $R$ in $ R \otimes M_0 $.
  Then, $ \kappa_M $ is $ \cmdB $-linear.
\end{lemma}

\begin{proof}
  Let $ \cmdHom_{\cmdB} ( R , M ) $ denote the vector space of all $ \cmdB $-linear maps $ \varphi : R \rightarrow M $.
  By (\ref{eq3.7}), $ \varphi \mapsto \varphi ( 1 ) $ gives the isomorphism
  \begin{equation} \label{eq3.9}
    \cmdHom {}_{\cmdB} ( R , M ) \cmdarrow^{\simeq} M_0 .
  \end{equation}
  Moreover, $ \kappa_M $ is identified with the evaluation map $ R \otimes \cmdHom_{\cmdB} ( R , M ) \rightarrow M $, which is obviously $ \cmdB $-linear.
\end{proof}

\begin{theorem} \label{3.10}
  Under the same assumptions as in Theorem \ref{3.6}, we have the following.
  \begin{enumerate}
    \item $ R_0 = k $.
    \item $ \kappa_M $ is injective for every left $ \cmdB $-module $M$.
  \end{enumerate}
\end{theorem}

\begin{proof}
  (1) By Theorem \ref{3.6}, $R$ is simple.
  By Schur's Lemma, $ D := \cmdEnd_{\cmdB} ( R ) $ is a division algebra,
  and the image $ \cmdIm \rho $ of $ \rho $ is included in $ \cmdEnd_D ( R ) $, which implies that $ \cmdEnd_D ( R ) $ is dense in $ \cmdEnd ( R ) $.
  But, $D$ then must equal $k$, which proves (1) by (\ref{eq3.9}).

  (2) Since $R$ is simple and $ k = \cmdEnd_{\cmdB} ( R ) $,
  it follows by \cite[Proposition 3.1]{AM} (or \cite[Proposition 12.5]{AMT}) that the evaluation map cited in the last proof above is injective, which proves (2).
\end{proof}

\subsection{}

In what follows until the end of this section we assume that $V$, $W$ are both finite-dimensional so that $R$, $S$ are locally finite.

\begin{definition} \label{3.11}
  A left $ \cmdB $-module $M$ is said to be {\rm integrable},
  if each element $ m \in M $ is annihilated by $ S ( n ) $, 
  where $ n ( > 0 ) $ is a sufficiently large integer that may depend on $m$.
  In the category of all left $ \cmdB $-modules, let $ \cmdO ( \cmdB ) $ denote the full subcategory consisting of the integrable modules;
  this is closed under sub, quotient and direct sum, whence it is abelian.
\end{definition}

The next lemma follows by (\ref{eq3.6}).

\begin{lemma} \label{3.12}
  The left $ \cmdB $-module $R$ is integrable.
  Moreover, if $X$ is a vector space, the left $ \cmdB $-module $ R \otimes X $, in which $ \cmdB $ acts on the factor $R$, is integrable.
\end{lemma}

Let $ \cmdVec $ denote the category of the vector spaces (over $k$).
We have functors,
\begin{eqnarray}
  ( \quad )_0 &:& M \mapsto M_0 , \quad \cmdO ( \cmdB ) \rightarrow \cmdVec, \label{eq3.10} \\
  R \otimes &:& X \mapsto R \otimes X , \quad \cmdVec \rightarrow \cmdO ( \cmdB ). \label{eq3.11}
\end{eqnarray}

\begin{theorem} \label{3.13}
  Assume that $ \tau_1 $ is non-degenerate, and $R$, $S$ are both Nichols algebras.
  Then the functors $ ( \quad )_0 $, $ R \otimes $ are quasi-inverses of each other,
  so that we have a category equivalence $ \cmdO ( \cmdB ) \approx \cmdVec $.
\end{theorem}

We have the natural maps
\begin{eqnarray*}
  \iota_X &:& X \rightarrow ( R \otimes X )_0 , \quad \iota_X ( x ) = 1 \otimes x, \\
  \kappa_M &:& R \otimes M_0 \rightarrow M, \mbox{ as given by (\ref{eq3.8}). }
\end{eqnarray*}
in $ \cmdVec $, $ \cmdO ( \cmdB ) $, respectively.
To prove the theorem above, notice that under (half of) the assumptions,
$ \iota_X $ is an isomorphism, and $ \kappa_M $ is injective, by Theorem \ref{3.10}.
It then remains to prove that $ \kappa_M $ is surjective.

\subsection{}

In this subsection we do not assume the assumptions above,
while we keep to assume that $V$, $W$ are finite-dimensional.

We will use the obvious identification $ \cmdB = \bigoplus_{n = 0}^{\infty} R \otimes S ( n ) $.
Set
\begin{equation*}
  \cmdI_n = \bigoplus_{l \geq n} R \otimes S ( l ) \quad ( n = 0 , 1 , \cdots ) 
\end{equation*}
in $ \cmdB $; this is the left ideal generated by $ S ( n ) $.
Regard $ \cmdB $ as a topological vector space which has $ \cmdI_n ( n = 0 , 1 , \cdots ) $ as a basis of neighborhoods of $0$.
Set
\begin{equation*}
  \hat{\cmdB} = \prod_{n = 0}^{\infty} R \otimes S ( n ) , \quad \hat{ \cmdI }_n = \prod_{l \geq n} R \otimes S ( l ) \quad ( n = 0 , 1 , \cdots ) .
\end{equation*}
These are the completions of $ \cmdB $, $ \cmdI_n $, respectively,
and their topologies coincide with the direct product of the discrete topologies on $ R \otimes S ( n ) $.

\begin{proposition} \label{3.14}
  \begin{enumerate}
    \item $ \cmdB $ is a topological algebra, whence as its completion,
    $ \hat{\cmdB} $ turns into a complete topological algebra.
    \item Let $ M \in \cmdO ( \cmdB ) $.
    Then the $ \cmdB $-module structure $ \cmdB \times M \rightarrow M $ uniquely extends to $ \hat{ \cmdB } \times M \to M $ so that each $ m \in M $ is annihilated by $ \hat{ \cmdI }_n $,
    where $ n ( > 0 ) $ is some integer that may depend on $m$.
    By this extended structure, $M$ is a left $ \hat{ \cmdB } $-module.
  \end{enumerate}
\end{proposition}

\begin{proof}
  (1) Let $ \beta = \sum_{i} r_i \# s_i \in \cmdB $,
  with $ r_i \in R $, $ s_i \in S $ homogeneous.
  One then sees that for each $ n \geq 0 $, $ \beta \mathcal{I}_n \subset \mathcal{I}_n $, $ \mathcal{I}_n \beta \subset \mathcal{I}_{n - l} $,
  where $ l = \cmdMax \{ \cmddeg r_i \} $.
  It follows that the product on $ \cmdB $ is continuous, which proves (1).

  (2) Suppose that $M$ is discrete.
  The integrability is equivalent to that for each $ m \in M $,
  the map $ \cmdB \rightarrow M $ given by $ \beta \mapsto \beta m $ is continuous.
  The map has a completion $ \hat{ \cmdB } \rightarrow M $;
  let $ \hat{ \beta } m $ denote the image of $ \hat{ \beta } \in \hat{ \cmdB } $.
  We then see that $ ( \hat{ \beta } , m ) \mapsto \hat{ \beta } m $ gives the desired structure.
\end{proof}

By Part (2) above,
$ \cmdO ( \cmdB ) $ is now naturally identified with the category of those left $ \hat{ \cmdB } $-modules $M$ in which every element $ m \in M $ is annihilated by some $ \hat{ \cmdI }_n $.

\subsection{}

Here we assume that the assumptions of Theorem \ref{3.13} are satisfied,
or equivalently that $ \tau_{R S} : R \otimes S \rightarrow k $ is non-degenerate; see Proposition \ref{2.6}.
By Corollary \ref{2.4} (2), this last is equivalent to that $ \tau_{ R ( n ) S ( n ) } : R ( n ) \otimes S ( n ) \rightarrow k $ is non-degenerate for each $ n \geq 0 $.
Choose bases $ ( r_{i_n} ) $ of $ R ( n ) $, and $ ( s_{i_n}) $ of $ S ( n ) $ which are mutually dual with respect to $ \tau_{ R ( n ) S ( n ) } $.

\begin{definition} \label{3.15}
  Let $ \cmdS $ denote the antipode of the braided graded Hopf algebra $R$, and define
  \begin{equation*}
    \gamma = \sum_{n = 0}^{\infty} \sum_{ i_n } \cmdS ( r_{i_n} ) \# s_{i_{n}} \mbox{ in } \hat{\cmdB}.
  \end{equation*}
  We call this $ \gamma $ the {\rm extremal projector}.
\end{definition}

The definition above is independent of choice of dual bases,
as will be seen from the proof of Part (1) of the proposition below.

We remark that the antipode $\cmdS$ of $R$ is bijective.
This follows from the formula given in \cite[Proposition 2 (b)]{R}, since the antipode of $J$, and hence that of $\mathcal{A}$ are both bijective.

\begin{proposition} \label{3.16}
  \begin{enumerate}
    \item $ \gamma $ is an idempotent in $ \hat{\cmdB} $ such that
    \begin{equation*}
      S ( n ) \gamma = 0 = \gamma R ( n ) \mbox{ for all } n > 0.
    \end{equation*}
    \item The infinite sum $ \sum_{n = 0}^{\infty} \sum_{i_n} r_{i_n} \gamma s_{i_n} $
    converges in $ \hat{\cmdB} $ to the unit $1$.
    \item Let $ M \in \cmdO ( \cmdB ) $. Then,
    \begin{equation*}
      \gamma \cdot : M \rightarrow M_0 , \quad m \mapsto \gamma m
    \end{equation*}
    gives a projection from $ M $ onto the subspace $ M_0 $.
  \end{enumerate}
\end{proposition}

\begin{proof}
  (1) Since $ \tau_{R S}^r : S \rightarrow R^g $ is now an isomorphism,
  we see from the proof of Theorem \ref{3.6} that $ \rho : \cmdB \rightarrow \cmdEnd ( R ) $ is injective.
  Moreover, this is continuous, and its completion gives an isomorphism
  \begin{equation*}
    \hat{\rho} : \hat{\cmdB} \cmdarrow^{\simeq} \cmdEnd ( R )
  \end{equation*}
  of topological algebras.
  We see
  \begin{eqnarray*}
    \hat{\rho} ( \gamma ) ( r ) &=& \sum_{n = 0}^{\infty} \sum_{i_n} \cmdS ( r_{i_n} ) \tau ( r_{( 1 )}, s_{i_n} ) r_{( 2 )} \\
                                &=& \cmdS ( r_{( 1 )} ) r_{( 2 )} = \varepsilon ( r ) 1,
  \end{eqnarray*}
  so that $ \hat{\rho} ( \gamma ) $ equals the projection $ R \rightarrow R $ given by $ r \mapsto \varepsilon ( r ) 1 $.
  This proves (1).
  
  (2) This follows, since $ r_{i_n} \gamma s_{i_n} \in \hat{\cmdI}_n $, and we see
  \begin{equation*}
    \sum_{n = 0}^{\infty} \sum_{i_n} \hat{\rho} ( r_{i_n} \gamma s_{i_n} ) ( r ) = \sum_{n = 0}^{\infty} \sum_{i_n} r_{i_n} \tau ( r_{( 1 )} , s_{i_n} ) \varepsilon ( r_{( 2 )} ) = r.
  \end{equation*}
  
  (3) If $ m \in M $, $ s \in S ( 1 ) $, then $ s ( \gamma m ) = ( s \gamma ) m  = 0 $ by (1), whence $ \gamma m \in M_0 $.
  If $ m \in M_0 $, then it is annihilated by all $ s_{i_n} $, where $ n > 0 $, and hence $ \gamma m = m $.
  These two prove (3).
\end{proof}

\begin{proof}[Proof of Theorem 3.13]
  To prove the remaining surjectivity of $ \kappa_M $, let $ m \in M \ ( \in \cmdO ( \cmdB ) ) $.
  Then by Part (2) above,
  \begin{equation} \label{eq3.12}
    m = \sum_{n = 0}^{\infty} \sum_{i_n} r_{i_n} \gamma s_{i_n} m .
  \end{equation}
  Notice that by the integrability, this infinite sum is essentially finite.
  Since $ \gamma s_{i_n} m \in M_0 $ by Part (3) above, we see $ m \in R M_0 $,
  which proves the desired surjectivity.
\end{proof}

\begin{remark} \label{3.17}
  In the situation of Theorem \ref{3.13}, assume that one of $R$, $S$ is finite-dimensional.
  Then both of them, and so $ \cmdB $ are all finite-dimensional.
  Moreover one sees from the proof of Theorem \ref{3.6} that $ \rho $ gives an isomorphism $ \displaystyle \cmdB \cmdarrow^{\simeq} \cmdEnd ( R ) $,
  whence $ \cmdB $ is a matrix algebra, up to isomorphism.
  The category $ \cmdO ( \cmdB ) $ consists of all left $ \cmdB $-modules, and the category equivalence $ \cmdO ( \cmdB ) \approx \cmdVec $ given by Theorem \ref{3.13} coincides with the familiar one.
\end{remark}

\section{Generalized $q$-boson algebras associated to generalized Kac-Moody algebras}

Throughout this section we assume that the characteristic of $k$ is zero, and let $ q \in k \setminus 0 $.

\subsection{}

We assume that $q$ is transcendental over $ \cmdbbQ $.
Let $I$ be a countable index set, $ \cmdbbA = ( a_{i j} )_{ i , j \in I } $ a symmetrizable Borcherds-Cartan matrix,
and $ \cmdbbm = ( m_i \mid i \in I ) $ a sequence of positive integers.
Kang \cite{Kan} defined the quantized enveloping algebra $ U = U_q $ associated to $ \cmdbbA $, $ \cmdbbm $.
We refer to \cite{KT} for the definition, and let the symbols
\begin{equation*}
  s_i , \ \alpha_i , \ ( h \mid h' ) , \ \xi_i , \ q^h , \ e_{i k} , \ f_{i k} , \ K_i , \ U^0 , \ U^{\pm} , \ U^{\geq 0} , \ U^{\leq 0}
\end{equation*}
be the same as given in \cite[pp. 348-350]{KT}.
In particular, the zero part $ U^0 $ of $U$ is the group algebra $ k P^{\vee} $ of the group $ P^{\vee} $ of the dual weight lattice as given in \cite[(1.1)]{KT}.
To apply the results obtained in the preceding sections,
we will work in the special situation that $ J = K = U^0 $.
Set
\begin{equation*}
  f'_{i k} = - \xi_i f_{i k} K_i \quad ( \in U^{\leq 0} ),
\end{equation*}
and let $V$ denote the subspace of $ U^{\leq 0} $ spanned by all $ f'_{i k} $.
Let $W$ denote the subspace of $ U^{+} $ spanned by all $ e_{i k} $.
Then, $V$, $W$ turn into objects in $ {}_{ U^0 }^{ U^0 } \cmdYD $ with respect to the structures
\begin{eqnarray*}
  q^h \cmdrhu e_{i k}  &=& q^{ \alpha_i ( h ) } e_{i k} , \quad e_{i k} \mapsto K_i \otimes e_{i k} , \\
  q^h \cmdrhu f'_{i k} &=& q^{ - \alpha_i ( h ) } f'_{i k} , \quad f'_{i k} \mapsto K_i \otimes f'_{i k} ,
\end{eqnarray*}
where $ h \in P^{\vee} $, $ i \in I $, $ 1 \leq k \leq m_i $.
Let $R$ denote the subalgebra of $ U^{\leq 0} $ generated by $V$.
Let $S$ stand for $ U^+ $, which is by definition the subalgebra of $ U^{\geq 0} $ generated by $W$.
We then see that $R$, $S$ are naturally pre-Nichols algebras of $V$, $W$, respectively, so that
\begin{equation*}
  R \cmddotrtimes U^0 = U^{\leq 0} , \quad S \cmddotrtimes U^0 = U^{\geq 0}.
\end{equation*}
Notice that if $ \overline{ \cmdS } $ denotes the pode of $ U^{\leq 0} $,
\begin{equation*}
  \overline{\cmdS} ( f'_{i k} ) = \xi_i f_{i k} , \quad ( \overline{R} = ) \overline{\cmdS} ( R ) = U^- .
\end{equation*}
By \cite[Proposition 2.1]{KT} (or our Proposition \ref{2.1}), the pairing $ \tau_0 : U^0 \otimes U^0 \rightarrow k $ given by
\begin{equation*}
  \tau_0 ( q^h , q^{h'} ) = q^{ - ( h \mid h' ) } \quad ( h , h' \in P^{\vee} )
\end{equation*}
uniquely extends to a skew pairing $ \tau : U^{\leq 0} \otimes U^{\geq 0} \rightarrow k $ so that
\begin{eqnarray*}
  & \tau ( q^h , e_{i k} ) = 0 = \tau ( f'_{i k} , q^h ) , & \\
  & \tau ( f'_{j l} , e_{i k} ) = \delta_{i j} \delta_{k l} &
\end{eqnarray*}
where $ h \in P^{\vee} $, $ i , j \in I $, $ 1 \leq k \leq m_i $, $ 1 \leq l \leq m_j $.
As is seen from the remark following Proposition \ref{2.1} (or seen directly),
the restriction $ \tau_1 : V \otimes W \rightarrow k $ of $ \tau $ satisfies (\ref{eq2.1}), (\ref{eq2.2}), and is obviously non-degenerate.
Let $ \cmdB $ denote the associated generalized $q$-boson algebra.
Then, $ \cmdB $ is generated by all $ f'_{i k} $, $ e_{i k} $, and is defined by the relations (R5) - (R8) in \cite[p.349]{KT},
in which $ f_{i k} $ should read $ f'_{i k} $, and by
\begin{equation*}
  e_{i k} f'_{j l} = q_{i j} f'_{j l} e_{i k} + \delta_{i j} \delta_{k l},
\end{equation*}
where $ i , j \in I $, $ 1 \leq k \leq m_i $, $ 1 \leq l \leq m_j $, and we have set
\begin{equation} \label{eq4.1}
  q_{i j} = q^{ - s_i a_{i j} } ( = q^{ - s_j a_{j i} } ) .
\end{equation}
By \cite[Theorem 2.5]{KT}, the skew pairing $ \tau $ above is non-degenerate,
whence $R$, $S$ are Nichols algebras, by Proposition \ref{2.6} (1).

As was seen in \ref{3.2}, $R$ has a natural left $ \cmdB $-module structure,
in which each $ f'_{i k} $ in $ \cmdB $ acts by the left multiplication.
One sees from (\ref{eq3.5}) that $ e_{i k} $ acts so as
\begin{equation} \label{eq4.2}
  e_{i k} \triangleright f'_{j_1 l_1} \cdots f'_{j_t l_t} = \sum_{p = 1}^t \delta_{i j_p} \delta_{k l_p} q_{i j_1} \cdots q_{i j_{p - 1}} f'_{j_1 l_1} \cdots \widehat{f}_{j_p l_p}' \cdots f'_{j_t l_t} ,
\end{equation}
where $ q_{i j} $ is such as given in (\ref{eq4.1}), and $ \ \widehat{} \ $ means an omission.
By Theorems \ref{3.6}, \ref{3.10}, we have the following.

\begin{theorem} \label{4.1}
  \begin{enumerate}
    \item $R$ is simple as a left $ \cmdB $-module.
    \item Given a left $ \cmdB $-module $M$, set
    \begin{equation*}
      M_0 = \{ m \in M \mid e_{i k} m = 0 \mbox{ for all } i \in I , 1 \leq k \leq m_i \}.
    \end{equation*}
    Then,
    \begin{equation*}
      \kappa_M : R \otimes M_0 \rightarrow M , \quad \kappa_M ( r \otimes m ) = r m
    \end{equation*}
    is a $ \cmdB $-linear injection.
    \item In particular, $ R_0 = k $.
  \end{enumerate}
\end{theorem}

\subsection{}

Keep $ \cmdB $ as above. Since $V$, $W$ can be infinite-dimensional, we need to modify the definition of integrable $ \cmdB $-modules.

\begin{definition} \label{4.2}
  A left $ \cmdB $-module $M$ is said to be {\rm integrable},
  if for each $ m \in M $, there exist an integer $ n > 0 $ and a finite subset $ F \subset I $ such that $m$ is annihilated by the product $ e_{i_1 k_1} \cdots e_{i_t k_t} $ of length $ t > 0 $,
  if $ t \geq n $, or if $ i_p \not \in F $ for some $ 1 \leq p \leq t $.
  In the category of all left $ \cmdB $-modules, let $ \cmdO ( \cmdB ) $ denote the full subcategory consisting of the integrable modules;
  this is abelian, being closed under sub, quotient and direct sum.
\end{definition}

If the index set $I$ is finite, the definition above coincides with Definition \ref{3.11}.

\begin{lemma} \label{4.3}
  The left $ \cmdB $-module $R$ is integrable.
\end{lemma}

\begin{proof}
  This is seen from (\ref{4.2}).
\end{proof}

\begin{theorem} \label{4.4}
  We have a category equivalence
  \begin{equation*}
    \cmdO ( \cmdB ) \approx \cmdVec,
  \end{equation*}
  which is given by the mutually quasi-inverse functors $ R \otimes $, $ ( \quad )_0 $ defined by (\ref{eq3.10}), (\ref{eq3.11}).
\end{theorem}

\begin{remark} \label{4.5}
To prove this theorem in the same way of the proof of Theorem \ref{3.13}, we remark that the argument in \ref{3.5}, \ref{3.6} can be modified, as follows, so as to fit in with the present situation.
We may suppose that the index set $I$ is (countably) infinite since if it is finite, we do not need any modification.

(1) Given a subset $ F \subset I $, we define
  \begin{equation*}
    \cmdbbA_F = ( a_{i j} )_{i, j \in F} , \quad \cmdbbm_F = ( m_i \mid i \in F )
  \end{equation*}
  with the restricted index set.
  Since $ \cmdbbA_F $ is still a symmetrizable Borcherds-Cartan matrix,
  we have the quantized enveloping algebra $ U_F $, say, associated to $ \cmdbbA_F $, $ \cmdbbm_F $.
  Let $ R_F $, $ S_F $, $ \cmdB_F $ denote the associated objects corresponding to $ R , S , \cmdB $, respectively.
  Notice that $ R_F \subset R $, $ S_F \subset S $ , and so $ \cmdB_F \subset \cmdB $, since $ R_F $, $ S_F $ are Nichols.
  If $F$ is finite, the completion $ \hat{ \cmdB }_F $ of $ \cmdB_F $ is defined so as in \ref{3.5}.

(2) Recall that $S(1)$ ($=W$) has a basis, $(e_{ik}\;|\;i\in I,\ 1\leq k\leq m_{i})$.
For a subset $F\subset I$, the subobject $S_{F}(1)$ of $S(1)$ in ${}_{ U^0 }^{ U^0 }\cmdYD$ has a basis, $(e_{ik}\;|\;i\in F,\ 1\leq k\leq m_{i})$.
We have a projection $\pi_{F} : S(1)\rightarrow S_{F}(1)$ given by
\[ \pi_{F}(e_{ik})=\left\{\begin{array}{ll} e_{ik} & \mbox{if $i\in F$}, \\ 0 & \mbox{otherwise.} \end{array}\right. \]
This uniquely extends to a projection $S\rightarrow S_{F}$ onto the braided graded Hopf subalgebra $S_{F}\subset S$, which we call $\tilde{\pi}_{F}$.
One sees that the kernel $\cmdKer \tilde{\pi}_{F}$ is the ideal of $S$ generated by all $e_{ik}$ with $i\not\in F$.

(3) In what follows we restrict the subsets $F\subset I$ in our consideration, only to finite ones.
The pairs $(n,F)$ of an integer $n\geq 0$ and a finite subset $F\subset I$ form a directed set with respect to the natural order which is defined so that $(n,F)\leq (l,G)$ if $n\leq l$ and $F\subset G$.
Given such a pair $(n,F)$, define a left ideal of $\cmdB$ by
\[ \cmdI_{n,F}:=\bigoplus_{l\geq n}R\otimes S(l)+R\otimes\cmdKer\tilde{\pi}_{F}. \]
Notice that if $(n,F)\leq (l,G)$, then $\cmdI_{n,F}\supset\cmdI_{l,G}$, whence $(\cmdI_{n,F})$ form a projective system.
We can regard $\cmdB$ as a topological vector space which has $(\cmdI_{n,F})$ as its basis of neighborhoods of $0$.
By modifying the proof of Proposition \ref{3.14}, we see the following.
\begin{enumerate}
\renewcommand{\labelenumi}{(\roman{enumi})}
\item $\cmdB$ is a topological algebra, whence its completion $\hat{\cmdB}$ is a complete topological algebra.
\item An object in $\cmdO(\cmdB)$ is precisely a left $\cmdB$-module $M$ such that each element $m\in M$ is annihilated by some $\cmdI_{n,F}$.
Such a module in turn is identified with a left $\hat{\cmdB}$-module $M$ such that each element $m\in M$ is annihilated by the completion $\hat{\cmdI}_{n,F}$ of some $\cmdI_{n,F}$.
\end{enumerate}

(4) We claim that the linear representation $\rho : \cmdB\rightarrow\cmdEnd(R)$ is continuous, and its completion gives an isomorphism $\hat{\rho} : \hat{\cmdB}\xrightarrow{\simeq}\cmdEnd(R)$ of complete topological algebras.
To see this, notice that $\hat{\cmdB}$ is the projective limit of $\prod_{n=0}^{\infty}R\otimes S_{F}(n)$ along the projections
\[ \prod_{n=0}^{\infty}R\otimes S_{G}(n)\rightarrow\prod_{n=0}^{\infty}R\otimes S_{F}(n)\quad (F\subset G\subset I) \]
analogously defined as the $\tilde{\pi}_{F}$ above, while $\cmdEnd(R)$ is the projective limit of $\cmdHom(R_{F},R)$ along the restriction maps $\cmdHom(R_{G},R)\rightarrow\cmdHom(R_{F},R)$.
The claim then follows since we see as in the proof of Theorem \ref{3.6} that the linear representations $\cmdB_{F}\rightarrow\cmdEnd(R_{F})$, with $S\otimes_{S_{F}}$ applied, give rise to isomorphisms
\[ \prod_{n=0}^{\infty}R\otimes S_{F}(n)\xrightarrow{\simeq}\cmdHom(R_{F},R) \]
of complete topological vector spaces, which are compatible with the projective systems on both sides and whose projective limit coincides with $\hat{\rho}$.
We define the {\em extremal projector} $\gamma$ in $\hat{\cmdB}$ to be a unique element such that $\hat{\rho}(\gamma)(r)=\varepsilon(r)1$ for every $r\in R$.
By this definition, $\gamma$ has the same property as described in Part (1) of Proposition \ref{3.16}.
We see that this $\gamma$ is characterized as the element which for every finite $F\subset I$, is mapped by the projection $\hat{\cmdB}\rightarrow\prod_{n=0}^{\infty}R\otimes S_{F}(n)$ to the extremal projector in $\hat{\cmdB}_{F}$ as defined by Definition \ref{3.15}.

(5) For each integer $n\geq 0$, we can choose a basis $(s_{p}\;|\;p\in\mathfrak{X}(n))$ of $S(n)$ so that each $s_{p}$ is a monomial in $e_{ik}$ of length $n$.
For a finite subset $F\subset I$, let $\mathfrak{X}_{F}(n)$ denote the set of those indicies $p\in\mathfrak{X}(n)$ for which $s_{p}$ is a monomial only in $e_{ik}$ with $i\in F$.
Then, $(s_{p}\;|\;p\in\mathfrak{X}_{F}(n))$ is a basis of $S_{F}(n)$.
Set
\[ \mathfrak{X}=\bigsqcup_{n=0}^{\infty}\mathfrak{X}(n),\qquad \mathfrak{X}_{n,F}=\bigsqcup_{l<n}\mathfrak{X}_{F}(l). \]
Then we have
\[ \cmdB=\bigoplus_{p\in\mathfrak{X}}R\otimes s_{p},\quad
\cmdI_{n,F}=\bigoplus_{p\not\in\mathfrak{X}_{n,F}}R\otimes s_{p},\quad
\hat{\cmdB}=\prod_{p\in\mathfrak{X}}R\otimes s_{p},\quad
\hat{\cmdI}_{n,F}=\prod_{p\not\in\mathfrak{X}_{n,F}}R\otimes s_{p}. \]
Let $\cmdS$ denote the antipode of $R$.
Since it is bijective (see the remark just above Proposition \ref{3.16}), there uniquely exists a family $(r_{p}\;|\;p\in\mathfrak{X})$ of elements in $R$ such that
\[ \gamma = \sum_{p\in\mathfrak{X}}\cmdS(r_{p})\otimes s_{p}. \]
The characterization of $\gamma$ noted last in (4) above proves that for each pair $(n,F)$, $(r_{p}\;|\;p\in\mathfrak{X}_{F}(n))$ is the basis of $R_{F}(n)$ dual to $(s_{p}\;|\;p\in\mathfrak{X}_{F}(n))$, and so that $(r_{p}\;|\;p\in\mathfrak{X})$ is a (unique) basis of $R$ such that $\tau(r_{p},s_{q})=\delta_{pq}$.
Obviously, $\gamma$ has the same property as described in Part (3) of Proposition \ref{3.16}.
The statement of the remaining Part (2) is modified so that the sequence
\[ \sum_{p\in\mathfrak{X}_{n,F}}r_{p}\gamma s_{p}\ \ (\in\cmdB) \]
directed by the pairs $(n,F)$ converges in $\hat{\cmdB}$ to the unit $1$.
\end{remark}

\begin{proof}[Proof of Theorem \ref{4.4}]
We may suppose as above that $I$ is infinite.
To prove the remaining surjectivity of $\kappa_{M}$, let $m\in M$ ($\in\cmdO(\cmdB)$).
Then the last statement of Remark \ref{4.5} proves
\[ m = (\lim_{n,F}\sum_{p\in\mathfrak{X}_{n,F}}r_{p}\gamma s_{p})m = \lim_{n,F}(\sum_{p\in\mathfrak{X}_{n,F}}r_{p}\gamma s_{p}m). \]
By the integrability this last limit equals the finite sum $\sum_{p\in\mathfrak{X}_{n,F}}r_{p}\gamma s_{p}m$, where $(n,F)$ is large enough.
The desired surjectivity follows since $\gamma s_{p}m\in M_{0}$ by the property $S(n)\gamma = 0$ ($n>0$); see Remark \ref{4.5} (4).
\end{proof}

\begin{remark} \label{4.6}
  Tan \cite[Proposition 3.2(3)]{Tan} asserts essentially the same result as above, in the slightly specialized situation when all $ m_i $ equal $1$.
  But, the argument of his proof, asserting that $ X g ( u_2 ) = X $ on Page 4346, line-3, is wrong.
  His definition of integrability given in \cite[p.4343, lines 3-5]{Tan} is incomplete, I think.
\end{remark}

\subsection{}

Suppose in particular that the index set $I$ is finite, $ \cmdbbA $ is a symmetrizable generalized Cartan matrix, and all $ m_i $ equal $1$.
Let $ e_i $, $ f_i $ denote the present generators $ e_{i 1} $, $ f_{i 1} $ of $ U^+ $, $ U^- $, respectively.
In this specialized situation the result essentially the same as our last theorem was first announced by Kashiwara \cite[Remark 3.4.10]{K},
and then proved by Nakashima \cite[Theorem 6.1]{N},
who introduced the extremal projector in his situation; cf. Definition \ref{3.15}. 
More precisely, Nakashima worked on the algebra $ B = B_q ( \cmdg ) $, which is bigger than our $ \cmdB $, including $ U^0 $, too.
We remark that his $B$ coincides with the $ ( U , U' ) $-bicleft extension over $k$ which was given by \cite[Proposition 4.9]{M1},
where $U'$ denote the (graded) Hopf algebra defined just as $U$,
but the relation $ [ e_i , f_j ] = \delta_{i j} \xi_i^{-1} ( K_i - K_i^{-1} ) $ is replaced by $ [ e_i , f_j ] = 0 $.
We also remark that Nakashima's \cite[Theorem 6.1]{N} can be reformulated as a category equivalence between $ \cmdO ( B ) $ and the category of all $P$-graded vector spaces,
where $P$ denotes the group of the weight lattice; see \cite[p.286]{N}.

\subsection{}

Suppose further that the generalized Cartan matrix $ \cmdbbA $ is of finite type.
We choose a diagonal matrix $ \cmdbbD = \cmddiag ( s_i \mid i \in I ) $ which makes $ \cmdbbD \cmdbbA $ symmetric,
in the standard manner so that $ 1 \leq s_i \leq 3 $, in particular.
Suppose that $q$ is a root of $1$ whose order $N$, say, is odd, and is not divided by $3$ if the Dynkin diagram of $ \cmdbbA $ includes a connected component of type $ G_2 $.
The Frobenius-Lusztig kernel $ u = u_q $ is the finite-dimensional quotient Hopf algebra of $ U = U_q $ subject to those relations which say that the $N$-th powers of $ q^h ( h \in P^{\vee} ) $ and of the positive root vectors in $ U^{\pm} $ are $1$ and $0$, respectively; see \cite{L}.
Let $ u^0 $, $ u^{\geq 0} $, $ u^{\leq 0} $, $ R' $, $ S' $ denote the natural images of $ U $, $ U^{\geq 0} $, $ U^{\leq 0} $, $R$, $S$, respectively, in $u$.
It is known that $ R' $, $ S' $ are Nichols algebras in $ {}_{u^0}^{u^0} \cmdYD $ so that $ R' \cmddotrtimes u^0 = u^{\leq 0} $, $ S' \cmddotrtimes u^0 = u^{\geq 0} $.
We see that the skew pairing $ \tau $ above factors through $ u^{\leq 0} \otimes u^{\geq 0} \rightarrow k $,
which is a non-degenerate skew pairing.
Therefore, to the associated generalized $q$-boson algebra now of finite dimension, Remark \ref{3.17} can apply.

\section*{Acknowledgments}

The author would thank the refree for his or her helpful comments which contain especially a hint to improve the exposition of Subsection 4.2.


\end{document}